\documentclass{amsart}
\usepackage{graphicx,euscript,amsmath}
\usepackage[alphabetic]{amsrefs}

\newtheorem{theorem}{Theorem}
\newtheorem{lemma}[theorem]{Lemma}

\theoremstyle{remark}
\newtheorem{remark}{Remark}

\newcommand{\myfig}[3][]{
 \begin{figure}
 \begin{center}
 {\mbox{\includegraphics[#1]{#2.eps}}}
 \end{center}
 \caption{\label{#2}#3}
 \end{figure}}

 \newcommand{\myfigh}[3][]{
 \begin{figure}[h!]
 \begin{center}
 {\mbox{\includegraphics[#1]{#2.eps}}}
 \end{center}
 \caption{\label{#2}#3}
 \end{figure}}

\def\vc#1{\mathbf #1}

\def\O{{\mathcal O}}

\long\def\ignore#1{}

\def\cc{\gamma}

\def\H{{\mathcal H}}

\def\nd{\noindent}
\def\fref#1{Figure~\ref{#1}}

\def\ss{\sigma}

 \providecommand{\Pic}{\mathop{\rm Pic}\nolimits}
 \providecommand{\Aut}{\mathop{\rm Aut}\nolimits}

\def\Bbb#1{{\mathbb #1}}

\def\D{{\mathcal D}}
\def\E{{\mathcal E}}

\def\dd{\delta}

\def\CC{{\Gamma}}

\def\K{{\mathcal K}}

\def\L{{\mathcal L}}

\def\hh{\hat h}
\def\ll{\lambda}

\begin{document}

\title{The Neron-Tate pairing and elliptic K3 surfaces}
\author{Arthur Baragar}
\begin{abstract}
In this paper, we demonstrate a connection between the group structure and Neron-Tate pairing on elliptic curves in an elliptic fibration with section on a K3 surface, and the structure of the ample cone for the K3 surface.  Part of the result can be thought of as a case of the specialization theorem.
\end{abstract}
\address{Department of Mathematical Sciences, University of Nevada, Las Vegas, NV 89154-4020}
\email{baragar@unlv.nevada.edu}
\thanks{\nd \LaTeX ed \today.}

\maketitle

\section*{Introduction}

In an earlier work \cite{Bar11}, we described a hyperbolic cross section of the ample cone for a class of K3 surfaces, and came up with the Poincar\'e ball model in \fref{fig1a}.  If $X=X(\Bbb Q)$ is a K3 surface in this class, then $X$ is fibered by elliptic curves, all of which are in the divisor class represented by the divisor $[E]$ in \fref{fig1a}.  If we unfold the Poincar\'e ball into the Poincar\'e upper half space with $[E]$ the point at infinity, we get \fref{fig2c}.  The Euclidean structure of the boundary at infinity is rather captivating, and is related to the group structure of the elliptic curves in the fibration.

\myfigh{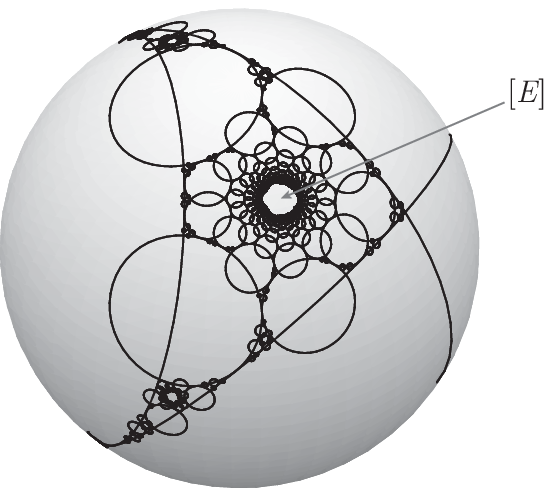}{A hyperbolic cross section of the ample cone for a class of K3 surfaces, rendered in the Poincar\'e ball model.}

\myfig{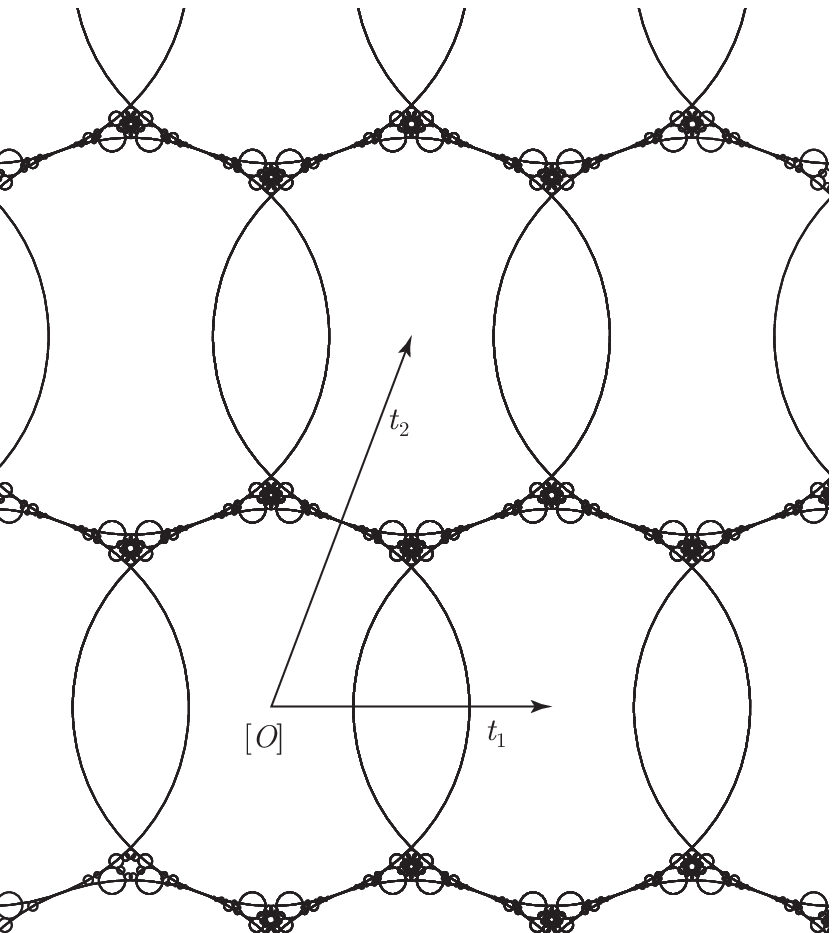}{The hyperbolic cross section of the same ample cone, unfolded into the upper half space model with $[E]$ the point at infinity.}

Let the Euclidean lattice we see in \fref{fig2c} be generated by the horizontal translation $t_1$ and diagonal translation $t_2$.
The fibration has a section $O$ with divisor class $[O]$.  The plane $[O]\cdot \vc x=0$ is a face of the ample cone, and is represented by the circle so noted in \fref{fig2c}.
Since $O$ is a smooth rational curve, $[O]\cdot [O]=-2$, and because it is a section, $[O]\cdot [E]=1$.
That is, for every elliptic curve $E$ in the divisor class $[E]$, there is a unique point $O_E$ in the intersection of $E$ with the curve $O$.
The section gives us a natural way of defining the zero element $O_E$ on every elliptic curve in the fibration.

The translates $D_1=t_1([O])$ and $D_2=t_2([O])$ give two more divisors with self intersection $-2$.  Since the corresponding faces ($D_i\cdot \vc x=0$) are walls of the ample cone, they both are represented with $-2$ curves \cite{Kov94}.  Since $t_i$ fixes $E$, we have $D_i\cdot [E]=t_i([O])\cdot [E]=[O]\cdot t_i^{-1}[E]=[O]\cdot [E]=1$, so both are sections.  For fixed $E$ in the fibration $[E]$, let $Q_{i,E}=D_i\cap E$, where we have abused notation by letting $D_i$ represent both a divisor and the unique $-2$ curve in the divisor class.  Let us use the points $Q_{i,E}$ (as $E$ varies in $[E]$) to define automorphisms $\tau_i\in \Aut(X/\Bbb Q)$ by
\[
\tau_i(P)=P+Q_{i,E},
\]
where $E$ is the fiber that contains $P$.  Then, as we will see in this paper, $\tau_{i*}=t_i$.  That is, the translation $t_i$ represents addition by $Q_{i,E}$ on the elliptic curve $E$.

Let $\vc v_i=D_i-[O]$.  Then $t_i$ restricted to the Euclidean plane (the boundary of the hyperbolic space) would appear to be translation in the direction $\vc v_i$, and this is indeed the case as we will see.  

Let $D$ be an ample divisor and $h_D$ a Weil height on $X$ associated to $D$.  For a fixed $E$ in $[E]$, we define the canonical height $\hat h$ on $E$ using $h_D$ restricted to $E$.  The Neron-Tate pairing on a pair of points $P_1$ and $P_2$ in the group $Q_{1,E}\Bbb Z\oplus Q_{2,E}\Bbb Z$ is given by
\begin{equation}\label{eq1}
\langle P_1,P_2\rangle = \hat h(P_1+P_2)-\hat h(P_1)-\hat h(P_2).
\end{equation}
Let us define the height of $E$ to be $h(E)=h_D(O_E)$.  Then, as we will see in this paper,
\[
\langle Q_{i,E},Q_{j,E}\rangle=h(E)([E]\cdot D)\vc v_i\cdot \vc v_j+O(1).
\]
The error term represented by $O(1)$ is independent of $E$.  (We trust that there is no confusion between the notation for the section $O$ and the big Oh notation.)  In particular
\[
\lim_{h(E)\to \infty} \frac{1}{h(E)([E]\cdot D)}\langle Q_{i,E},Q_{j,E}\rangle=\vc v_i\cdot \vc v_j,
\]
which gives the relative geometry of the lattice in \fref{fig2c}.

This phenomenon is true in general.

\begin{theorem} \label{t1} Let $X/k$ be a K3 surface defined over a number field $k$ with an elliptic fibration $[E]$ and a section $O$.  Let
\[
\CC=\{\ss_*: \ss\in \Aut(X/k)\}
\]
be the pull back of the group of automorphisms on $X$.  Suppose $\CC_{[E]}$, the stabilizer of $[E]$, has an Abelian subgroup $G=t_1\Bbb Z \oplus ... \oplus t_{\rho-2}\Bbb Z\cong \Bbb Z^{\rho-2}$ of maximal rank.  Let $D$ be an ample divisor on $X$ and $h_D$ a Weil height associated to $D$.  Let $D_i$ represent both the divisor class $t_i([O])$ and the unique $-2$ curve in $D_i$.  For any elliptic curve $E$ in the fibration $[E]$, let
\begin{align*}
O_E&=O\cap E \\
Q_{i,E}&= D_i\cap E \\
\vc v_i&=D_i-[O]\\
h(E)&=h_D(O_E).
\end{align*}
Define $\tau_i: X\to X$ by
\[
\tau_i(P)=P+Q_{i,E},
\]
where $E$ is the elliptic curve in $[E]$ that contains $P$.
Finally, let $\langle , \rangle$ be the Neron-Tate pairing on $E$ for any $E\in [E]$.  Then
\[
\tau_{i*}=t_i
\]
and
\begin{equation}\label{eq2}
\lim_{h(E)\to \infty} \frac{1}{h(E)([E]\cdot D)}\langle Q_{i,E},Q_{j,E}\rangle= \vc v_i\cdot \vc v_j.
\end{equation}
Furthermore, the map $t_i$ is translation by $\vc v_i$ when viewed in a Poincar\'e upper half hyperspace model with $[E]$ the point at infinity.  
\end{theorem}

Some of the preceding, and in particular Eq.~(\ref{eq2}), appears in works by Silverman and Tate \cites{Tat83, Sil83} (see also \cite{Sil94}*{Section III.11}).  Our approach and the geometric interpretation, via pictures of the ample cone, appear to be novel.  Lemmas \ref{l2.1} and \ref{l2.2} are also notable and possibly novel.

\begin{remark}
The notation of this example was chosen so as to be consistent with a suitable notation for this paper, and differs significantly from the notation used in \cite{Bar11}.  For those who might be interested, $[E]$ is denoted by $D_1$ in \cite{Bar11}, $[O]$ by $D_4$, $t_1=ST_2ST_2$, and $t_2=T_2T_4$.
\end{remark}

\section{Background}
\subsection{K3 surfaces}
Let $X$ be a K3 surface defined over a number field $k$.  Let $\Pic(X)$ be its Picard group and let $\{D_1,...,D_{\rho}\}$ be a basis over $\Bbb Z$, where $\rho$ is the Picard number.  Let $J=[D_i\cdot D_j]$ be the intersection matrix.  By the Hodge Index Theorem, $J$ has signature $(1,\rho-1)$, so is a Lorentz product on $\Pic(X)\otimes \Bbb R$, and hence there is an underlying hyperbolic structure.  Let $D$ be an ample divisor and define the light cone to be
\[
\L=\{\vc x\in \Pic(X)\otimes \Bbb R: \vc x\cdot \vc x>0, \vc x\cdot D>0\}.
\]
Let
\[
\H=\{\vc x\in \L: \vc x \cdot \vc x=1\}.
\]
For two points $A$ and $B$ in $\L$, let us define a distance $|AB|$ by $||A||||B||\cosh |AB|=A\cdot B$, where $||\vc x||=\sqrt{\vc x\cdot \vc x}$.  Then the set $\H$ equipped with this distance is a model of hyperbolic geometry $\Bbb H^{\rho-1}$.  At times, it will be convenient to identify $\H$ with $\L/\Bbb R^+$, equipped with the same metric.  The boundary $\partial \H = \partial \L/\Bbb R^+$ is the usual compactification of $\Bbb H^{\rho-1}$ and is congruent to $\Bbb S^{\rho-2}$.

Let $\K$ be the ample cone for $X$.  A cross section of $\K$ is a polyhedron with possibly an infinite number of faces.  Each face is a plane through the origin, so forms a hyperplane in $\H$.  This is what the rendering in \fref{fig1a} represents:  Every circle on the sphere represents a plane in the Poincar\'e ball model of $\Bbb H^2$, and $\K/\Bbb R^+$ is the region bounded by all these hyperbolic planes.

Let
\[
\O(R)=\{T\in M_{2\times 2}(R): T^tJT=J\},
\]
and
\[
\O^+(R)=\{T\in \O(R): T\L=\L\}.
\]
Then $\O^+(\Bbb R)$ is the group of isometries on $\H$.  Let $\O''\leq \O^+(\Bbb Z)$ be the group of symmetries of $\K$ in $\O^+(\Bbb Z)$.  If $\ss\in \Aut(X/k)$, then its pullback $\ss^*$ clearly preserves $\K$, has integer entries, and preserves the intersection pairing.  We therefore have a natural homomorphism
\begin{align*}
\Phi: \Aut(X/k) &\to \O'' \\
\ss&\mapsto \ss^*.
\end{align*}
For $k$ sufficiently large, the map $\Phi$ has a finite kernel and co-kernel \cite{PS-S}.

\subsection{The Euclidean structure of $\partial \L/\Bbb R^+$}\label{s1.2}
Let $\Bbb R^{1,\rho-1}$ be a Lorentz space equipped with the Lorentz product $\cdot$ (which may be thought of as $\Pic(X)\otimes \Bbb R$ equipped with the intersection pairing).  The superscript $1,\rho-1$ is the signature of the Lorentz product; that is, it has one positive eigenvalue and $\rho-1$ negative eigenvalues.  Let us distinguish a $D$ with $D\cdot D>0$ and define the light cone $\L$ as above.  Let us fix $E\in \partial \L$ and define
\[
\partial H_E:=(\partial \L\setminus E\Bbb R^+)/\Bbb R^+.
\]
For any $A\in \partial \L\setminus E\Bbb R^+$, let $\bar A$ be its equivalence class in $\partial \H_E$.  For any $\bar A$ and $\bar B\in \partial \H_E$, let us define
\[
|\bar A\bar B|_{E}:=\sqrt{\frac{2A\cdot B}{(A\cdot E)(B\cdot E)}}.
\]

\begin{lemma}  The function $|\bar A\bar B|_{E}$ defines a Euclidean metric on $\partial \H_E$.  Furthermore, if $\cc$ preserves the Lorentz product and $\cc E=E$, then $\cc$ is a Euclidean isometry on $\partial \H_E$.
\end{lemma}

\begin{proof}  We first note that $|\bar A\bar B|_{E}$ is invariant under scalar multiples of $A$ or $B$, so it is well defined on $\partial \H_E$.

Let us define the space perpendicular to $E$ to be
\[
V^{\perp E}:= \{\vc x\in \Bbb R^{1,\rho-1}: \vc x\cdot E=0\}.
\]
Note that $E\in V^{\perp E}$.

If $\vc x\in V^{\perp E}$ and $\vc x \cdot \vc x=0$, then the space spanned by $\vc x$ and $E$ is in $\partial \L$.  Since $\partial \L$ is a cone, it contains no two-dimensional subspaces, so $\vc x$ must be a scalar multiple of $E$.  Thus, $V^{\perp E}$ is tangent to $\partial \L$.  Now suppose $P\in \partial \L$ but is not a multiple of $E$. Then $P$ and $D$ are on the same side of $V^{\perp E}$, so $P\cdot E$ and $D\cdot E$ share the same sign.  That is, $P\cdot E>0$.  Without loss of generality, we may scale $P$ with a positive scalar so that $P\cdot E=1$.

Since $P\notin V^{\perp E}$, the set $\{P,V^{\perp E}\}$ spans $\Bbb R^{1,\rho-1}$.  For an arbitrary $A\in \partial \L$, let us write
\[
A=a_PP+a_EE+\vc a,
\]
where
\[
\vc a\in V^{\perp E,P}:=\{\vc x\in \Bbb R^{1,\rho-1}: \vc x\cdot E=\vc x\cdot P=0\}.
\]
As with $P$, we may scale $A$ so that $A\cdot E=1$.  Note that $A\cdot E=a_P$, so we now have
\begin{equation}\label{eq1.1}
A=P+a_EE+\vc a.
\end{equation}
Since $A\cdot A=0$, we get
\[
0=A\cdot A=(P+a_EE+\vc a)\cdot (P+a_EE+\vc a)=2a_E+\vc a\cdot \vc a,
\]
so
\[
a_E=\frac{-\vc a\cdot \vc a}2.
\]

Let us now calculate $|\bar A\bar B|_{E}$:
\begin{align*}
|\bar A\bar B|_{E}&=\sqrt{\frac{2(P-\frac{\vc a\cdot \vc a}2E+\vc a)\cdot (P-\frac{\vc b\cdot \vc b}2 E+\vc b)}{(A\cdot E) (B\cdot E)}} \\
&=\sqrt{-\vc a\cdot \vc a - \vc b\cdot \vc b +2 \vc a \cdot \vc b} \\
&=\sqrt{-(\vc a-\vc b)\cdot (\vc a-\vc b)}.
\end{align*}
We note that $V^{\perp E,P}$ is the intersection of two tangent spaces to the light cone, so $\vc a$ and $\vc b$ are in a space where the Lorentz product is negative definite.  Thus $|\bar A\bar B|_{E}$ is a Euclidean metric.

Finally, if $\cc$ preserves the Lorentz product and $\cc E=E$, then $\cc$ clearly preserves the metric $|\bar A\bar B|_{E}$, so is a Euclidean isometry on $\partial \H_E$.
\end{proof}

The Euclidean structure outlined above is the one we are used to; that is to say, it is the $(\rho-2)$-dimensional Euclidean structure of the boundary of the Poincar\'e upper-half space model of $\Bbb H^{\rho-1}$ with $E$ the point at infinity.

The Euclidean space $\partial \H_E$ can also be represented by $V^{\perp E,P}$ in a natural way, via the identification
\begin{align*}
\phi: \qquad  \partial \H_E &\to V^{\perp E,P} \\
\bar A &\mapsto \frac{\vc a}{A\cdot E},
\end{align*}
where $A$ is any representative of $\bar A$ and $A=a_PP+a_EE+\vc a$, with $\vc a\in V^{\perp E,P}$.  Note that 
\[
\vc a= A-\frac{A\cdot E}{E\cdot P}P-\frac{A\cdot P}{E\cdot P}E.
\]

We can use this subspace to build a Poincar\'e upper half-space model of $\Bbb H^{\rho-1}$.   We let $(\vc x,z)\in V^{\perp E,P}\times \Bbb R^+$ and equip this set with the arclength element $ds$ where
\[
ds^2=\frac{-d\vc x\cdot d\vc x+dz^2}{z^2}.
\]
The negative sign arises because the Lorentz product is negative definite on $V^{\perp E,P}$.

\begin{lemma}  Let $\vc U=wP+vE +\vc u\in \H$ where $\vc u\in V^{\perp E,P}$.  The map
\begin{align*}
\Phi: \qquad \H &\to V^{\perp E,P} \times \Bbb R^+ \\
  \vc U &\mapsto \left(\frac{\vc u}{w},\frac{1}{w}\right)
\end{align*}
is an isomorphism of hyperbolic spaces.
\end{lemma}

\begin{proof}  We prove this by demonstrating that the arclength element $ds'$ on $\H$ induced by the Lorentz product maps to the Poincar\'e arclength element $ds$.  Let $\Phi(\vc U)=(\vc x,z)$.  Then
\begin{alignat*}{2}
\vc x&= \frac{\vc u}{w} \qquad & z&=\frac{1}w \\
d\vc x&=\frac{d\vc u}w -\frac{\vc udw}{w^2} \qquad & dz&=-\frac{dw}{w^2}.
\end{alignat*}
Thus,
\[
ds^2=-d\vc u\cdot d\vc u+\frac{2(\vc u\cdot d\vc u) dw}{w}-\frac{(\vc u\cdot\vc u) dw^2}{w^2}+\frac{dw^2}{w^2}.
\]
Using $2vw+\vc u\cdot \vc u=1$, we get
\[
ds^2=-d\vc u\cdot d\vc u+\frac{2\vc u\cdot d\vc u dw}{w}+\frac{2vwdw^2}{w^2}.
\]
The arclength element induced by the Lorentz product satisfies
\[
(ds')^2=-d\vc U\cdot d\vc U=-2dvdw-d\vc u\cdot d\vc u,
\]
where the minus sign comes from the signature $(1,\rho-1)$ of our Lorentz product.
We use $2vw+\vc u\cdot \vc u=1$ to solve for $dv$:
\begin{align*}
2vdw +2wdv+2\vc u\cdot d\vc u&=0 \\
-dv&=\frac{vdw+\vc u\cdot d\vc u}{w},
\end{align*}
and plugging this into our formula for $(ds')^2$:
\[
(ds')^2=\frac{2vdw^2}{w}+\frac{2(\vc u\cdot d\vc u) dw}{w}-d\vc u\cdot d\vc u=ds^2,
\]
as desired.
\end{proof}

Suppose $\CC_E$ has a subgroup $G$ isomorphic to $\Bbb Z^{\rho-2}$.  Since $\CC$ is discrete, and $\CC_E$ acts as a group of Euclidean isometries on $\partial \H_E$, we know the elements of $G$ when restricted to $V^{\perp E,P}$ are translations.  The following result describes such maps.

\begin{lemma}\label{l1.3}  Let $\vc v\in V^{\perp E,P}$.  The map
\[
T_{\vc v}(\vc x)=\vc x -\left(\vc x\cdot \vc v+\frac{1}2 (\vc x\cdot E)(\vc v\cdot \vc v)\right)E+(\vc x\cdot E)\vc v
\]
is in $\O^+$, fixes $E$, and acts as translation by $\vc v$ in $V^{\perp E,P}\cong \partial \H_E$.
\end{lemma}

\begin{proof}  It is a straightforward calculation to verify that $T_{\vc v}(\vc x)\cdot T_{\vc v}(\vc y)=\vc x\cdot \vc y$, and that $T_{\vc v}(E)=E$.  If $A\in \partial \H_E$, then $A\cdot A=0$, and after some calculation, one finds $\phi(T_{\vc v}(A))=\phi(A)+\vc v$.  Hence the action of $T_{\vc v}$ on $V^{\perp E,P}$ is translation by $\vc v$.
\end{proof} 

Note that $([O]+[E])\cdot ([O]+[E])=0$ and $([O]+[E])\cdot [E]=1$, so we may choose $P=[O]+[E]$.  With this choice, $t_i=T_{\vc v}$ for some $\vc v \in V^{\perp [E],[O]+[E]}$.  Note that
\[
D_i=t_i([O])=T_{\vc v}([O])=[O]+cE+\vc v
\]
for some $c$.  We can isolate $c$ by noting $[O]\cdot \vc v= ([O]+[E])\cdot \vc v=0$ so $D_i\cdot [O]=-2+c$.  Thus $\vc v=D_i-[O]-(2+D_i\cdot [O])E$.  Thus $t_i=T_{\vc v}$.  Note that $\vc v_i$ and $\vc v$ differ by a multiple of $E$.  While Lemma \ref{l1.3} uses $\vc v\in V^{\perp E,P}$, it is straightforward to verify that $T_{\vc v}=T_{\vc v+aE}$ for any $a$.  Thus, $t_i=T_{\vc v_i}$, as desired.   

\begin{remark}  The map $T_{\vc v}$ was derived by first considering an arbitrary $A=a_PP+a_EE+\vc a\in \partial\L$ with $\vc a\in V^{\perp E,P}$.  We note that $T_{\vc v}(A)\cdot E=A\cdot T_{\vc v}^{-1}E=A\cdot E=a_P$, so
\[
T_{\vc v}(a_P^{-1}A)=P+a_E'E+\vc a+\vc v.
\]
We use $T_{\vc v}(A)\cdot T_{\vc v}(A)=0$ to solve for $a_E'$.  Finally, we note that $\vc a\cdot \vc v=A\cdot \vc v$ and gather together the components of $A$ to get the formula for $T_{\vc v}$.
\end{remark}

\begin{remark}  It is straight forward to verify $T^m_{\vc v}(\vc x)=T_{m\vc v}(\vc x)$, and that $T_{\vc v}\circ T_{\vc w}=T_{\vc w}\circ T_{\vc v}$.
\end{remark}

\subsection{The Neron-Tate Pairing}

Let $E$ be an elliptic curve defined over a number field $k$.  Then $E(k)\cong E_{{\rm Tor}}\times \Bbb Z^r$ and there exists a basis $\{P_1,...,P_r\}$ to the torsion-free part of $E(k)$.  For a point $P\in E(k)$, let $H(P)$ be its naive height.  The logarithmic height of $P$ is $h(P)=\log(H(P))$, and the canonical height is
\[
\hh(P)=\lim_{n\to \infty}\frac{h([n]P)}{n^2}.
\]
The canonical height has several nice properties:
\begin{align}
h(P)&=\hh(P)+O(1) \label{eqNT1} \\
\hh([n]P)&=n^2\hh(P) \label{eqNT2} \\
\hh(P+Q)+\hh(P-Q)&=2\hh(P)+2\hh(Q). \label{eqNT3}
\end{align}
From Eq.~(\ref{eqNT1}) and (\ref{eqNT2}), it follows that $\hh(P)=0$ if and only if $P\in E_{{\rm Tor}}$.  We define the Neron-Tate pairing to be
\[
\langle P,Q\rangle = \hh(P+Q)-\hh(P)-\hh(Q).
\]
It is a nice exercise (using Eq.~(\ref{eqNT3})) to show that the Neron-Tate pairing is a bilinear form.  That is, $\langle P,Q\rangle=\langle Q,P\rangle$ and
\[
\langle [m]P+Q,R\rangle=m\langle P, R\rangle+\langle Q, R\rangle.
\]

\subsection{Vector Heights}  Given a basis $\D=\{D_1,...,D_\rho\}$ of $\Pic(X)$, let us define a dual basis $\D^*=\{D_1^*,...,D_\rho^*\}$ such that
\[
D_i\cdot D_j^*=\dd_{ij},
\]
where $\dd_{ij}$ is the Kronecker-delta symbol ($\dd_{ij}=1$ if $i=j$, $\dd_{ij}=0$ if $i\neq j$).  For each $D_i$, let us pick a Weil height $h_{D_i}$ with respect to the divisor $D_i$, and define the vector height
\begin{align*}
\vc h(P): \qquad X &\to \Pic(X)\otimes \Bbb R \\
P &\mapsto \sum_{i=1}^\rho h_{D_i}(P)D_i^*.
\end{align*}
The vector height has a couple of nice properties \cite{Bar03}.  For any Weil height $h_D$ associated to the divisor $D$, we have
\[
h_D(P)=\vc h(P)\cdot D +O(1),
\]
where the constant implied by the $O(1)$ is independent of $P$, but may depend on $D$.  Also, for any $\ss\in \Aut(X)$,
\[
\vc h(\ss P)=\ss_*\vc h(P)+O(1),
\]
where again the constant implied by the $O(1)$ is independent of $P$ but may depend on $\ss$.

\section{The main result}

\subsection{The automorphisms of $X$ that fix the fibers in $[E]$}  Let
\begin{align*}
\ss_0: \qquad X& \to X \\
P &\mapsto -P
\end{align*}
where the operation is on the unique elliptic curve $E\in [E]$ that contains $P$.  Then $\ss_0\in\Aut(X)$ and its pullback $\ss_0^*$ acts linearly on $\Pic(X)$.  The following describes the action of $\ss_0^*$:

\begin{lemma}  \label{l2.1}  The pullback $\ss_0^*$ of $\ss_0$ has eigenvectors $[E]$ and $[O]$ associated to the eigenvalue $\ll=1$, and acts as multiplication by $-1$ on $V^{\perp [E],[O]}$.
\end{lemma}

\begin{proof}  Note that $\ss_0^2$ is the identity on $X$, so $(\ss_0^*)^2=\Bbb I$.  Hence the minimal polynomial for $\ss_0^*$ divides $\ll^2-1$, so $\ss_0^*$ is diagonalizable over $\Bbb Q$ (thinking of $\ss_0^*$ acting on $\Pic(X)\otimes \Bbb Q$) with eigenvalues $\ll=\pm 1$.  Since $\ss_0E=E$ for any $E\in [E]$, and $\ss_0O=O$, both $[E]$ and $[O]$ are eigenvectors with associated eigenvalue $\ll=1$.  The space $V^{\perp [E],[O]}$ perpendicular to $\hbox{span}\{[E],[O]\}$ is invariant under the action of $\ss_0^*$.  To see this, suppose $\vc v\cdot [E]=0$.  Then
\[
\ss_0^*\vc v\cdot [E]=\vc v \cdot \ss_0^*[E]=\vc v\cdot [E]=0.
\]
We can therefore complete a basis of eigenvectors (over $\Bbb Q$) with eigenvectors in $V^{\perp [E],[O]}$.  Suppose there exists an eigenvector $\vc w\in V^{\perp [E],[O]}$ with eigenvalue $\ll=1$.  Without loss of generality (by multiplying by a suitable integer), we may assume $\vc w$ is an integral linear combination of $\{\vc v_1,..., \vc v_{\rho-2}\}$.  (This is where we use that $\CC_E$ has an Abelian subgroup of maximal rank.)  Then
\[
T_{m\vc w}\in \CC_E
\]
for all $m\in \Bbb Z$.  Let $D_{m\vc w}=T_{m\vc w}([O])$.  Then $D_{m\vc w}\cdot \vc x=0$ is a face of the ample cone, so $D_{m\vc w}$ represents a $-2$ curve on $X$ for all integers $m$.  Further, by Lemma \ref{l1.3},
\[
D_{m\vc w}=T_{m\vc w}([O])=[O]+c_{m\vc w}E+m\vc w,
\]
so $D_{m\vc w}$ is an eigenvector of $\ss_0^*$ with eigenvalue $\ll=1$.  That is, $\ss_0^*D_{m\vc w}=D_{m\vc w}$, and hence $\ss_0 (D_{m\vc w})=D_{m\vc w}$, where we are abusing notation as before and letting $D_{m\vc w}$ represent both a divisor class and the unique $-2$ curve on $X$ that represents the class. Let
\[
Q_{m\vc w,E}=D_{m\vc w}\cap E\in X,
\]
for any $E\in [E]$.  Then
\[
\ss_0(Q_{m\vc w,E})=\ss_0(D_{m\vc w}\cap E)=\ss_0(D_{m\vc w})\cap \ss_0(E)=D_{m\vc w}\cap E=Q_{m\vc w,E}.
\]
But from the definition of $\ss_0$,
\[
\ss_0(Q_{m\vc w,E})=-Q_{m\vc w,E}.
\]
Hence, $2Q_{m\vc w,E}=0$.  There exists a generic fiber $E$ where the $-2$ curves $D_{m\vc w}$ intersect $E$ at infinitely many points, so on this fiber we have found an infinite number of points of order $2$, a contradiction.  Thus, there is no eigenvector in $V^{\perp [E],[O]}$  associated to the eigenvalue $\ll=1$, so$V^{\perp [E],[O]}$ is the eigenspace associated to $\ll=-1$.
\end{proof}

Let
\begin{align*}
\ss_i: \qquad & X\to X \\
P &\mapsto Q_{i,E}-P,
\end{align*}
where $E$ is the unique fiber in $[E]$ that contains $P$.  Then $\ss_i\in \Aut(X)$.

\begin{lemma} \label{l2.2} The pullback $\ss_i^*$ of $\ss_i$ has eigenvectors $[E]$ and $[O]+D_i$ associated to $\ll=1$, and is $-1$ on the perpendicular space $V^{\perp [E],[O]+D_i}$.
\end{lemma}

\begin{proof}  The proof is similar to that of Lemma \ref{l2.1}.  While $\ss_iE=E$ as before, we note that $\ss_iO=D_i$ and $\ss_iD_i=O$, so $\ss_i^*([O]+D_i)=[O]+D_i$.  As before, the perpendicular space $V^{\perp [E],[O]+D_i}$ is invariant under $\ss_i^*$, so we can complete a basis of eigenvectors with elements in this space.  We assume there exists an eigenvector $\vc w\in V^{\perp [E],[O]+D_i}$ (over $\Bbb Q$) with associated eigenvalue $\ll=1$.  Let us write
\[
\vc w=w_O[O]+w_E[E] + \vc w',
\]
where $\vc w'\in V^{\perp [E],[O]}$.  Then
\begin{align*}
0=\vc w\cdot [E]&=w_O \\
\vc w\cdot [O]&=w_E,
\end{align*}
so
\[
\vc w'=\vc w-(\vc w\cdot [O])[E].
\]
Thus, $\vc w'$ is in the eigenspace for $\ll=1$.  As before, by multiplying by a suitable integer, we may take $\vc w'$ to be an integer linear combination of $\{\vc v_1,...,\vc v_{\rho-2}\}$.  Thus, as before, $T_{m\vc w'}\in \CC_E$ for all integers $m$, $D_{m\vc w'}=T_{m\vc w'}([O])$ is represented by a $-2$ curve on $X$, $\ss_i^*D_{m\vc w'}=D_{m\vc w'}$, and therefore
\[
\ss_i(Q_{m\vc w',E})=  Q_{m\vc w',E}.
\]
But
\[
\ss_i(Q_{m\vc w',E})=Q_i-Q_{m\vc w',E},
\]
so we get $2Q_{m\vc w',E}=Q_i$ for all integers $m$.  For any $E$, there are at most four solutions to $2P=Q_i$, but for a generic fiber (all but finitely many), the points $Q_{m\vc w',E}$ form an infinite set.  Thus, no such $\vc w$ can exist, so $V^{\perp [E],[O]+D_i}$ is the $(\ll=-1)$-eigenspace for $\ss_i^*$.
\end{proof}

We are now ready to prove the first assertion of Theorem \ref{t1}.

\begin{theorem}  The map $t_i\in \CC_E$ is the push forward of $\tau_i\in \Aut(X)$.
\end{theorem}

\begin{proof}  We note that $\tau_i=\ss_i\circ \ss_0$.  Thus $\tau_i^*=\ss_0^*\circ \ss_i^*$, so
\[
\tau_{i*}=(\ss_0^*\circ \ss_i^*)^{-1}=\ss_i^*\circ \ss_0^*,
\]
where we have used that $\ss_i^2=\ss_0^2=id$.  We use Lemmas \ref{l2.1} and \ref{l2.2} to calculate $\tau_{i*}$.  Given $A\in \Pic(X)\otimes \Bbb Q$, let us write
\[
A=a_O[O]+a_E[E]+\vc a,
\]
where $\vc a\in V^{\perp [E],[O]}$.  We can write
\[
\vc a=\frac 12 \vc a\cdot ([O]+D_i)+\vc a',
\]
where $\vc a'\in V^{\perp [E],[O]+D_i}$, from which it follows
\[
\ss_i^*\circ \ss_0^* A=A+a_O(D_i-[O])-\vc a\cdot ([O]+D_i)[E].
\]
Noting that $a_O=A\cdot E$ and substituting $v_i$, we get
\[
\tau_{i*}(A)=A+(A\cdot E)\vc v_i+(2+D_i\cdot [O]-\vc a\cdot([O]+D_i))[E].
\]
Because $\CC$ is discrete and $t_i$ has infinite order in $\CC_E$, it must be a translation on $\E'$.  Thus $t_i=T_{\vc v}$ for some $\vc v\in V^{\perp [E],[O]}$.  Using $D_i=t_i([O])$ and Lemma \ref{l1.3}, we conclude $\vc v=\vc v_i$.  Thus,
\[
t_i(A)=A+(A\cdot E)\vc v_i +\left(A\cdot \vc v_i+\frac 12 (A\cdot [E])(\vc v_i\cdot \vc v_i)\right)[E].
\]
We can verify directly that the coefficients of $[E]$ in $\tau_{i*}(A)$ and $t_i(A)$ are equal, or we can observe that both $\tau_{i*}$ and $t_i$ are isometries, so
\begin{align*}
\tau_{i*}(A)\cdot D_i&=A\cdot\tau_{i*}^{-1}(D_i)=A\cdot [O] \\
t_i(A)\cdot D_i&=a\cdot t_i^{-1}(D_i)=A\cdot [O].
\end{align*}
Thus
\[
\tau_{i*}(A)\cdot D_i=t_i(A)\cdot D_i,
\]
and since $[E]\cdot D_i=1\neq 0$, we get that the coefficients of $[E]$ are equal.  Hence, $\tau_{i*}=t_i$, as claimed.
\end{proof}

\subsection{The Neron-Tate Pairing}  In this section, we use vector heights to calculate $\langle Q_{i,E},Q_{j,E}\rangle$.  Let us choose the basis
\[
\D=\{[E],[O]+[E],\vc v_1,...,\vc v_{\rho-2}\}.
\]
Then the dual basis is
\[
\D^*=\{[O]+[E],[E],\vc v_1^*,...,\vc v_{\rho-2}^*\}
\]
where $\hbox{span}\{\vc v_1^*,...,\vc v_{\rho-2}^*\}=V^{\perp [E],[O]}$.  Let us define a projection of $X$ onto the section $O$ by
\begin{align*}
\pi: \qquad  X &\to O \\
P &\mapsto O_E
\end{align*}
where $E$ is the unique fiber that contains $P$.  Let us define a logarithmic height $h$ on the section $O$.  Note that the pull back $\pi^*$ of a point in $O$ is a fiber in $[E]$, so $h\circ \pi$ is a Weil height with respect to $[E]$.    Let us use $h_{[E]}=h\circ \pi$ in our definition of the vector height, so $\vc h(P)\cdot [E]=h_{[E]}(P)$.  Now suppose $\ss\in \langle \ss_0,\ss_1,..., \ss_{\rho-2}\rangle$.  Since $\ss$ fixes $E$ for every $E\in [E]$, we know
\[
h_{[E]}(\ss P)=h_{[E]}(P).
\]
On the other hand,
\begin{align*}
h_{[E]}(\ss P)&=\vc h(\ss P)\cdot [E]=(\ss_*\vc h(P) +\vc O(1))\cdot [E] \\
&=\ss_*\vc h(P)\cdot [E]+\vc O(1)\cdot [E] \\
&=\vc h(P)\cdot \ss^*[E]+\vc O(1)\cdot [E] \\
&=\vc h(P)\cdot [E] +\vc O(1)\cdot [E] \\
&=h_{[E]}(P)+\vc O(1)\cdot [E].
\end{align*}
Thus, the error term $\vc O(1)$ for $\ss$ lies in $V^{\perp [E]}$, since it satisfies $\vc O(1)\cdot [E]=0$.

\begin{lemma}  Suppose $\vc u, \vc v \in V^{\perp [E]}$.  Then
\[
|\vc u\cdot \vc v| \leq | ||\vc u|| ||\vc v|| |,
\]
where $||\vc u||=\sqrt{\vc u\cdot \vc u}$.
\end{lemma}

\begin{proof}  Let us write $\vc u=u_E[E]+\vc u'$, etc., with $\vc u'\in V^{\perp [E],[O]}$.  Then
\[
\vc u\cdot \vc v=\vc u'\cdot \vc v',
\]
from which the result follows, since $\cdot$ is negative definite on $V^{\perp [E],[O]}$.
\end{proof}

Let $\vc v\in \vc v_1\Bbb Z\oplus ...\oplus \vc v_{\rho-2}\Bbb Z$.  Let $\tau_{\vc v}\in \Aut(X)$ be the canonical automorphism with $\tau_{\vc v *}=T_{\vc v}$.  (That is, $\tau_{\vc v}(P)=P+Q_{\vc v,E}$ where $E$ is the fiber in $[E]$ that contains $P$, $Q_{\vc v,E}=E\cap D_{\vc v}$ and $D_{\vc v}=T_{\vc v}([O])$.)  Then
\[
\vc h(\tau_{\vc v} P)=T_{\vc v}\vc h(P)+\vc O(1),
\]
where the error term is bounded.  Let us decompose the error terms into two parts, $\vc O(1)=\vc O'(1)+O(1)[E]$, where $\vc O'(1)\in V^{\perp [E],[O]}$.  Let $M$ bound both $| ||\vc O'(1)|| |$ and $|O(1)|$.  (Note that $M$ depends on $\vc v$, but not $P$.)  We drop the prime notation $\vc O'(1)$ in the following:

\begin{lemma}  For fixed $\vc v\in \vc v_1\Bbb Z\oplus ... \oplus \vc v_{\rho-2}\Bbb Z$,
\[
\vc h(\tau_{\vc v}^n P)=T_{\vc v}^n\vc h(P)+\vc O(n)+O(n^2)[E],
\]
where $\vc O(n)\in V^{\perp [E],[O]}$, $|||\vc O(n)|||$ is bounded by $Mn$ and the scalar error term $O(n^2)$ is bounded by $M|||\vc v|||n^2$.
\end{lemma}

\begin{proof}[Proof (by induction on $n$)]  The base case is clear.  Consider
\begin{align*}
\vc h(\tau_{\vc v}^{n+1}P)&=T_{\vc v}\vc h(\tau_{\vc v}^n P)+\vc O(1) +O(1)[E]\\
&=T_{\vc v}(T_{\vc v}^n\vc h(P)+\vc O(n)+O(n^2)[E])+\vc O(1) + O(1)[E] \\
&=T_{\vc v}^{n+1}\vc h(P)+T_{\vc v}(\vc O(n))+O(n^2)T_{\vc v}([E]) +\vc O(1)+O(1)[E] \\
&=T_{\vc v}^{n+1}\vc h(P)+\vc O(n)+(\vc O(n)\cdot \vc v)[E] +O(n^2)[E]+\vc O(1) +O(1)[E]\\
&=T_{\vc v}^{n+1}\vc h(P)+\vc O(n+1)+O((n+1)^2)[E].
\end{align*}
In the last step, we used the previous lemma to conclude $|\vc O(n)\cdot \vc v|\leq M|||\vc v|||$.
\end{proof}

We are now ready to calculate the canonical height for a point $Q_{\vc v,E}$ in the subgroup of $E[k]$ generated by the $Q_{i,E}$'s.   Let $D$ be ample and use the height $h_D(P)=\vc h(P)\cdot D$.  Then
\begin{align*}
\hat h(Q_{\vc v,E})&=\lim_{n\to \infty}\frac{h_D([n]Q_{\vc v,E})}{n^2} \\
&=\lim_{n\to\infty} \frac 1{n^2}\vc h(\tau_{\vc v}^nO_E)\cdot D \\
&=\lim_{n\to \infty} \frac 1{n^2}\left(T_{\vc v}^n\vc h(O_E)+\vc O(n)+O(n^2)[E]\right)\cdot D \\
&=\lim_{n\to\infty} \frac1{n^2}\bigg(\vc h(O_E)-\Big(\vc h(O_E)\cdot n\vc v+\frac 12 (\vc h(O_E)\cdot [E])(n\vc v\cdot n\vc v)\Big)[E] \\
&  \qquad \qquad \qquad \qquad+(\vc h(O_E)\cdot [E])\vc v+\vc O(n)+O(n^2)[E]\bigg)\cdot D \\
&=\frac 12 \left((\vc h(O_E)\cdot[E])(\vc v\cdot \vc v)+O(1)\right)([E]\cdot D) \\
&=\frac 12 h(E)(\vc v\cdot \vc v)([E]\cdot D)+O(1).
 \end{align*}
The error term $O(1)$ is bounded by $M|||\vc v|||([E]\cdot D)$, and is independent of our choice of fiber $E$.  This result is similar to \cite{Tat83}*{Corollary 1}.

Finally, we calculate the Neron-Tate pairing:
\begin{align*}
\langle Q_{\vc v,E},Q_{\vc w,E}\rangle &=\hat h(Q_{\vc v,E}+Q_{\vc w,E})-\hat h(Q_{\vc v,E})-\hat h(Q_{\vc w,E}) \\
&=\hat h(Q_{\vc v+\vc w,E})-\hat h(Q_{\vc v,E})-\hat h(Q_{\vc w,E}) \\
&=\frac 12 h(E)([E]\cdot D)\left((\vc v+\vc w)\cdot (\vc v+\vc w)-\vc v\cdot \vc v-\vc w\cdot \vc w\right)+O(1) \\
&=h(E)([E]\cdot D)(\vc v\cdot \vc w)+O(1).
\end{align*}
In particular,
\[
\lim_{h(E)\to \infty} \frac{1}{h(E)([E]\cdot D)}\langle Q_{\vc v,E},Q_{\vc w,E}\rangle=\vc v\cdot \vc w.
\]
These last two results are similar to \cite{Sil94}*{Theorem 11.1 and Corollary 11.3.1}.  

This completes all the pieces of Theorem \ref{t1}.

\end{document}